\documentclass[english,10pt]{amsart}
\usepackage[english]{babel}

\usepackage[matrix,arrow]{xy}
\xyoption{all}
\usepackage{amscd,amssymb,amsfonts,amsmath}

\usepackage{graphics}
\usepackage{epsfig}
\usepackage{mathrsfs}
\usepackage{xcolor}

\newcommand\mypagesizel{
\textwidth= 6.5in
\textheight=9in
\voffset-.55in
\hoffset -0.75in
\marginparwidth=56pt
}

\newcommand{\Pic}{\textup{Pic}}

\newcommand{\codim}{\textup{codim}}

\newcommand{\N}{\textup{N}}

\newcommand{\Mov}{\textup{Mov}}

\renewcommand{\phi}{\varphi}
\newcommand{\into}{\hookrightarrow}
\newcommand{\map}{\dashrightarrow}

\renewcommand{\le}{\leqslant}
\renewcommand{\ge}{\geqslant}

\newcommand{\rank}{\textup{rank}}

\newcommand{\bQ}{\mathbb{Q}}

\newcommand{\cQ}{\mathcal{Q}}

\newcommand{\sC}{\mathscr{C}}

\newcommand{\sE}{\mathscr{E}}
\newcommand{\sF}{\mathscr{F}}
\newcommand{\sG}{\mathscr{G}}

\newcommand{\sI}{\mathscr{I}}

\newcommand{\sK}{\mathscr{K}}
\newcommand{\sL}{\mathscr{L}}
\newcommand{\sM}{\mathscr{M}}
\newcommand{\sN}{\mathscr{N}}
\newcommand{\sO}{\mathscr{O}}

\newcommand{\sQ}{\mathscr{Q}}

\newcommand{\sT}{\mathscr{T}}

\mypagesizel

\newtheorem{thm}{Theorem}[section]
\newtheorem*{thm*}{Theorem}

\newtheorem{lemma}[thm]{Lemma}
\newtheorem{cor}[thm]{Corollary}

\theoremstyle{definition}
\newtheorem{defn}[thm]{Definition}

\newtheorem{say}[thm]{}

\newtheorem{ques}[thm]{Question}

\newtheorem{notation}[thm]{Notation}

\newtheorem{defn-thm}[thm]{Definition-Theorem}

\newtheorem{rem}[thm]{Remark}

\theoremstyle{remark}

\newtheorem*{not-and-def}{Notation and definitions}

\numberwithin{equation}{section}

\begin{document}

\title[]{Characterization of generic projective space bundles and algebraicity of foliations}

\author{Carolina \textsc{Araujo}} 

\address{\noindent Carolina Araujo: IMPA, Estrada Dona Castorina 110, Rio de
  Janeiro, 22460-320, Brazil} 

\email{caraujo@impa.br}

\author{St\'ephane \textsc{Druel}}
  
\address{St\'ephane Druel: Institut Fourier, UMR 5582 du CNRS, Universit\'e Grenoble Alpes, CS 40700, 38058 Grenoble cedex 9, France} 

\email{stephane.druel@univ-grenoble-alpes.fr}

\subjclass[2010]{14M22, 14J40, 37F75}

\begin{abstract}
In this paper we consider various notions of positivity for distributions on complex projective manifolds.
We start by analyzing distributions having big slope with respect to curve classes, 
obtaining characterizations of generic projective space bundles  in terms of  movable curve classes. 
We then apply this result to investigate algebraicity of leaves of foliations, providing 
a lower bound for the algebraic rank of a foliation in terms of invariants measuring positivity. 
We classify foliations attaining this bound, and 
describe those whose algebraic rank slightly exceeds this bound. 
\end{abstract}

\maketitle

\tableofcontents

%
%
%
%

\section{Introduction}

The existence of sufficiently positive subsheaves  of the tangent bundle of a complex projective manifold $X$
imposes strong restrictions on $X$. 
In particular, several special varieties can be characterized by positivity properties of their tangent bundle. 
Early results in this direction include Kobayashi and Ochiai's characterizations of projective spaces and hyperquadrics \cite{kobayashi_ochiai}
and Mori's characterization of projective spaces \cite{mori79}.
There are many ways of measuring positivity of a torsion free sheaf. 
One way is to consider slopes with respect to movable curve classes $\alpha\in\Mov(X)$.
We refer to Section~\ref{movable_classes} for the notion of slope and its properties.
When $\alpha$ is an ample class, we have the following characterization of projective spaces due to H\"oring. 

\begin{thm}\label{thm:hoering}
Let $X$ be an $n$-dimensional complex normal projective variety,  $\sL$  an ample line bundle on $X$,
and  $\sF \subseteq T_X$  a distribution. Set $\alpha:=[\sL^{n-1}]\in\Mov(X)$. 
If $\mu_\alpha(\sF\otimes\sL^*)>0$, then 
\begin{itemize}
	\item $(X,\sL)\cong (\mathbb{P}^n,\sO_{\mathbb{P}^n}(1))$ $(${\cite[Theorem 1.1]{hoering_fol}}$)$, and 
	\item $\sF=T_{\mathbb{P}^n}$ $(${\cite[Theorem 1.3]{adk08}}$)$.
\end{itemize}
\end{thm}

Projective space bundles provide counter-examples to the statement of Theorem~\ref{thm:hoering} if we replace the ample class 
$\alpha=[\sL^{n-1}]$ with  more general movable curve classes. More precisely, 
let $Y$ be a complex projective manifold, and $\sE$ an ample vector bundle of rank
$r+1 \ge 2$ on $Y$. Consider the projectivization $X=\mathbb{P}_Y(\sE)$, with tautological line bundle $\sO_X(1)$ and natural morphism
$\pi\colon X \to Y$. Let $\alpha \in \Mov (X)$ be the class of a line contained in a fiber of $\pi$.
Then $\mu_\alpha\big(T_{X/Y}\otimes\sO_X(-1)\big)=\frac{1}{r}>0$.
So there is an open neighborhood $U\subset \Mov (X)$ of $\alpha$ such that 
$\mu_\beta\big(T_{X/Y}\otimes\sO_X(-1)\big)>0$ for every $\beta\in U$. 
The following characterization of generic $\mathbb{P}^r$-bundles shows that these are all the new examples that arise
when the ample class $\alpha=[\sL^{n-1}]$ is replaced with an arbitrary movable class $\alpha\in\Mov(X)$ in Theorem~\ref{thm:hoering}.
By a \emph{generic $\mathbb{P}^r$-bundle} we mean an almost proper dominant map 
$X \dashrightarrow Y$ to a normal projective variety $Y$ with general fiber isomorphic to $\mathbb{P}^r$.

\begin{thm}\label{thm:p^-bdles}
Let $X$ be a normal $\mathbb{Q}$-factorial complex projective variety, and $\sL$ an ample line bundle on $X$.
If $\mu^{\max}_\alpha(T_X\otimes\sL^*)>0$ for some movable curve class 
$\alpha\in\Mov(X)$, then $X$ is a generic $\mathbb{P}^r$-bundle for some positive integer $r$.
\end{thm}

Theorem~\ref{thm:p^-bdles} follows from the more refined statement in Theorem~\ref{thm:positive_twisted_slope}.

\begin{cor}\label{cor:p^-bdles_singular_case} \label{cor:p^-bdles}
Let $X$ be a normal $\mathbb{Q}$-factorial complex  projective variety, and  suppose that $X$ is not a generic $\mathbb{P}^r$-bundle for any positive integer $r$. Let $\sL$ be an ample line bundle on $X$.
Then, for any positive integer $m$ and any  torsion-free  quotient  $(\Omega_X^1\otimes\sL)^{\otimes m}\twoheadrightarrow\sQ$ of positive rank, $\det(\sQ)$ is pseudo-effective. 
\end{cor}

Next we apply these results to   investigate algebraicity of leaves of holomorphic foliations on complex projective manifolds.

\medskip

A central problem in the theory of holomorphic foliations is to find conditions that guarantee the existence of algebraic leaves.
Algebraic leaves of holomorphic foliations correspond to algebraic solutions of complex differential equations. 
It has been noted that positivity of foliations on complex projective varieties 
tend to improve algebraicity of leaves (\cite{CLN}, \cite{bost}, \cite{bogomolov_mcquillan01}, \cite{campana_paun}, \cite{fano_fols}, \cite{lpt3fold}, \cite{codim_1_del_pezzo_fols}, \cite{fano_fols_2}, \cite{codim_1_mukai_fols}, \cite{campana_paun15}, \cite{druel15}, and \cite{bobo}). 
In order to measure algebraicity of leaves, we introduce the \emph{algebraic rank}. 
The algebraic rank  $r^a(\sF)$ of a holomorphic foliation $\sF$ on a complex algebraic variety $X$
is the maximum dimension of an algebraic subvariety  through a general point of $X$
that is tangent to $\sF$. 
These maximal algebraic subvarieties tangent to $\sF$ are the leaves of a subfoliation $\sF^a\subseteq \sF$, 
the \emph{algebraic part} of $\sF$ (see Definition~\ref{dfn:algebraic_part}).

In a series of papers, we have addressed \emph{Fano foliations} (\cite{fano_fols}, \cite{codim_1_del_pezzo_fols}, \cite{fano_fols_2}
and \cite{codim_1_mukai_fols}). 
These are holomorphic foliations $\sF$ on complex projective varieties 
with ample anti-canonical class  $-K_{\sF}$.
For a Fano foliation $\sF$ on a  complex projective manifold $X$,
a rough measure of positivity is the \emph{index} $\iota(\sF)$, which is the largest integer 
dividing $-K_{\sF}$ in $\Pic(X)$. 
Our works on Fano foliations with high index indicated that 
the larger is the index, the larger is the algebraic rank of the Fano foliation. 
Now we investigate this relation between positivity and algebraicity of leaves
for a wider class of foliations than Fano foliations.

We first introduce a new invariant measuring positivity for foliations with big anti-canonical class.

\begin{defn}\label{defn:generalized_index}
Let $X$ be a complex projective manifold, and $\sF$  a holomorphic foliation on $X$ with big anti-canonical class  $-K_{\sF}$.
The  \emph{generalised index} of $\sF$ is defined as follows (see Lemma~\ref{lemma:generalized_index}). 
$$
\widehat{\iota}(\sF):=\max\big\{t\in\mathbb{R} \ | \ -K_{\sF} \equiv tA+E \text{ where } A \text{ is an ample divisor and }
E \text{ is a pseudo-effective }\mathbb{R}\text{-divisor}\big\}.
$$
\end{defn}

We provide a lower bound for the algebraic rank in terms of the generalised index, and classify
foliations attaining this bound. We refer to Definition~\ref{dfn:algebraic_part} for the notions of 
pull-back and purely transcendental foliations.

\begin{thm}\label{thm:ko}
Let $X$ be an $n$-dimensional complex projective manifold, and $\sF\subsetneq T_X$ a foliation on $X$ with
big anti-canonical class. Then the algebraic rank and the generalised index of $\sF$ satisfy 
$$
r^a(\sF) \ \ge \ \widehat{\iota}(\sF).
$$
Moreover, equality holds if and only if $X \cong \mathbb{P}^n$ and 
$\sF$ is the linear pull-back of a purely transcendental foliation on $\mathbb{P}^{n-r^a(\sF)}$ with zero canonical class.
\end{thm}

The following are immediate consequences of Theorem \ref{thm:ko}.

\begin{cor}
Let $X$ be an $n$-dimensional complex projective manifold, and $\sF\subsetneq T_X$ a Fano foliation of index $\iota(\sF)$ on $X$. 
Then the algebraic rank of $\sF$ satisfies the following inequality.
$$
 r^a(\sF) \ \ge \ {\iota}(\sF).
$$
Moreover, equality holds if and only if $X \cong \mathbb{P}^n$ and 
$\sF$ is the linear pull-back of a purely transcendental foliation on $\mathbb{P}^{n-r^a(\sF)}$ with zero canonical class.
\end{cor}

\begin{cor}\label{cor:degree 0}
Let $X$ be an $n$-dimensional complex projective manifold,  $L$ an ample divisor on $X$,
and $\sF\subsetneq T_X$ a foliation of rank $r$ on $X$. 
Suppose that $-K_\sF-rL$ is pseudo-effective. 
Then $\big(X,\sO_X(L)\big)\cong \big(\mathbb{P}^n,\sO_{\mathbb{P}^n}(1)\big)$, and $\sF$ is induced by a linear projection of $\mathbb{P}^n$.
\end{cor}

Finally, we address foliations $\sF$ whose algebraic rank and  generalised index satisfy $r^a(\sF) \le  \widehat{\iota}(\sF) +1$.

\begin{thm}\label{thm:rat_connectedness_leaves}
Let $X$ be a complex projective manifold, and $\sF\subsetneq T_X$ a foliation on $X$  with
big anti-canonical class. 
Suppose that the algebraic rank and the generalised index of $\sF$ satisfy $r^a(\sF) \le  \widehat{\iota}(\sF) +1$.
Then the closure of a general leaf of $\sF^a$ is rationally connected.
\end{thm}

The following is an immediate consequence of Theorem \ref{thm:rat_connectedness_leaves}.

\begin{cor}
Let $X$ be a complex projective manifold, and $\sF\subsetneq T_X$ a Fano foliation of index $\iota(\sF)$ on $X$. 
Suppose that the algebraic rank of $\sF$ satisfies $r^a(\sF) = {\iota}(\sF) +1$.
Then the closure of a general leaf of $\sF^a$ is rationally connected.
\end{cor}

We remark that codimension $1$ Fano foliations $\sF$ of index $n-3$ on $n$-dimensional projective manifolds were classified in \cite{codim_1_mukai_fols}. 
When $X \not\cong \mathbb{P}^n$, then either $\sF$ is algebraically integrable or $r^a(\sF) = {\iota}(\sF) +1$.

\medskip

\noindent {\bf Notation and conventions.}
We always work over the field ${\mathbb C}$ of complex numbers. 
Varieties are always assumed to be irreducible.
We denote by $X_{\textup{ns}}$ the nonsingular locus of a variety $X$.
When $X$ is a normal variety, and $\sF$ is a quasi-coherent sheaf of generic rank $r$ on $X$, 
we denote by $T_X$ the sheaf $(\Omega_{X}^1)^*$, and by $\det(\sF)$ the sheaf $(\wedge^r\sF)^{**}$. 
If $\sE$ is a vector bundle on a variety $X$, we denote by
$\mathbb{P}_X(\sE)$ the Grothendieck projectivization $\textup{Proj}_X(\textup{S}^\bullet\sE)$.

\medskip

\noindent {\bf Acknowledgements.}
This project was supported by the Brazilian-French Network in Mathematics. 
Much of this work was developed during the authors' visits to IMPA and Institut Fourier.
We would like to thank both institutions for their support and hospitality. 

The authors would also like to thank the anonymous referees for their helpful and
very detailed reports.

The first named author was partially supported by CNPq and Faperj Research  Fellowships.
The second named author was partially supported by the ALKAGE project (ERC grant Nr 670846, 2015$-$2020) and
by the project Foliage of Agence Nationale de la Recherche (ANR-16-CE40-0008-01).

%
%
%
%

\section{Foliations} \label{section:foliations} 

\subsection{Basic notions}
Let $X$ be a normal variety. 

\begin{defn}\label{defn_dist}
A \emph{distribution} on $X$ is a coherent saturated subsheaf $\sF\subset T_X$. By saturated we mean that the quotient $T_X/\sF$ is torsion-free.

The \emph{rank} $r$ of $\sF$ is the generic rank of $\sF$.
The \emph{codimension} of $\sF$ is defined as $q:=\dim X-r$. 

The \emph{normal sheaf} of $\sF$ is the sheaf $\sN_\sF:=(T_X/\sF)^{**}$.

The \textit{canonical class} $K_{\sF}$ of $\sF$ is any Weil divisor on $X$ such that  $\sO_X(-K_{\sF})\cong \det(\sF)$. 
We say that $\sF$ is \emph{$\mathbb{Q}$-Gorenstein} if $K_\sF$ is a $\mathbb{Q}$-Cartier divisor.
\end{defn}

\begin{defn}\label{defn_foliation}
A \emph{foliation} on $X$ is a distribution  $\sF\subset T_X$  that is closed under the Lie bracket.

Let $X^\circ \subset X_{\textup{ns}}$ be the maximal open subset such that $\sF_{|X_{\textup{ns}}}$ is a subbundle of $T_{X_{\textup{ns}}}$. 
By Frobenius' Theorem, through any point of $X^\circ$ there is a maximal connected and immersed holomorphic submanifold $L \subset X^\circ$ such that
$T_L=\sF_{|L}$. We say that such $L$ is a  \emph{leaf} of $\sF$. A leaf is called \emph{algebraic} if it is open in its Zariski closure.
\end{defn}

\begin{say}\label{forms}
To a codimension $q$ distribution $\sF$ on a normal variety $X$, one naturally associates a unique (up to scaling) twisted $q$-form
$\omega_{\sF}\in H^0(X,\Omega^q_X\otimes \det(\sN_\sF))$. This form does not vanish in codimension $1$, and 
completely determines the distribution $\sF$.  (See for instance \cite[Paragraph 3.5]{fano_fols_2} for details.)
\end{say}

\begin{defn}[{The algebraic and transcendental parts of a  foliation  \cite[Definition 2]{codim_1_del_pezzo_fols}}]
\label{dfn:algebraic_part}
Let $\sF$ be a  foliation of rank $r$ on a normal variety $X$.
There exist  a normal variety $Y$, unique up to birational equivalence,  a dominant rational map with connected fibers $\varphi\colon X\map Y$,
and a  foliation $\sG$ on $Y$ such that the following conditions hold (see \cite[Section 2.3]{loray_pereira_touzet}).
\begin{enumerate}
	\item $\sG$ is purely transcendental, i.e., there is no positive-dimensional algebraic subvariety through a general point of $Y$ that is tangent to $\sG$.
	\item $\sF$ is the pull-back of $\sG$ via $\varphi$. This means the following. 
		Let $X^\circ\subset X$ and  $Y^\circ\subset Y$ be smooth open subsets such that $\varphi$ restricts to a smooth morphism 
		$\varphi^\circ\colon X^\circ\to Y^\circ$.
		Then $\sF_{|X^\circ}=(d\varphi^\circ)^{-1}(\sG_{|Y^\circ})$. In this case we write $\sF=\varphi^{-1}\sG$.
\end{enumerate}

The foliation $\sF^{a}$ on $X$ induced by $\varphi$ is called the \emph{algebraic part} of $\sF$, and its rank is the 
\emph{algebraic rank} of $\sF$, which we denote by $r^a$. When $r^a=r$, we say that $\sF$ is \emph{algebraically integrable}.
The foliation $\sG\subset T_Y$ is called the \emph{transcendental part} of $\sF$.
\end{defn}

Next we relate the canonical class of a foliation with those of its algebraic and transcendental parts. 
For that we introduce some notation.

\begin{defn}\label{defn_ramification}
Let $\pi \colon X \to Y$ be a dominant morphism of normal varieties. 

Let $D$ be a Weil $\mathbb{Q}$-divisor on $Y$. 
If $\pi \colon X \to Y$ is equidimensional, we define the pull-back $\pi^*D$ of $D$ to be the unique $\mathbb{Q}$-divisor on $X$ whose restriction to 
$\pi^{-1}(Y_{\textup{ns}})$ is $(\pi_{|\pi^{-1}(Y_{\textup{ns}})})^*D_{|\pi^{-1}(Y_{\textup{ns}})}$. 
This  agrees with the usual pull-back if $D$ is $\mathbb{Q}$-Cartier.

We define the \emph{ramification divisor} $R(\pi)$ of $\pi$ as follows. 
Let $Y^\circ\subset Y$ be a dense open subset such that $\codim(Y\setminus Y^\circ) \ge 2$ and  
$\pi$ restricts to an equidimensional morphism $\pi^\circ\colon X^\circ=\pi^{-1}(Y^\circ)\to Y^\circ$. 
Set 
$$
R(\pi^\circ)=\sum_{D^\circ} \Big((\pi^\circ)^*D^\circ-{\big((\pi^\circ)^*D^\circ\big)}_{\textup{red}}\Big),
$$
where $D^\circ$ runs through all prime divisors on $Y^\circ$. 
Then  $R(\pi)$ is the Zariski closure of $R(\pi^\circ)$ in $X$. 

Assume  that either $K_Y$ is $\mathbb{Q}$-Cartier, or that $\pi$ is equidimensional.
We define the \emph{relative canonical divisor} of $X$ over $Y$ as $K_{X/Y}:=K_X-\pi^*K_Y$. 
\end{defn}

\begin{say}[The canonical class of a pull-back foliation] \label{K_pullback}
Let $\pi\colon X\to Y$ be a dominant morphism with connected fibers between normal varieties, 
$\sG$ a  foliation on $Y$, and  $\sF=\pi^{-1}\sG\subset T_X$ its pull-back via $\pi$, as in Definition \ref{dfn:algebraic_part} (2).
Assume that $\pi$ is equidimensional.
Let  $\{B_i\}_{i\in I}$ be the (possibly empty) set of prime divisors on $Y$ contained in the set of critical values of $\pi$ and invariant by $\sG$. 
A straightforward computation shows that 
\begin{equation}
\label{pullback_fol}
K_{\sF} \ = \ \pi^*{K_{\sG}}+K_{X/Y}-
\sum_{i\in I}\big(\pi^*B_i-(\pi^*B_i)_{\textup{red}}\big).
\end{equation}
In particular, if $\sF$ is induced by $\pi$, then \eqref{pullback_fol} reads
\begin{equation}
\label{morphism_fol}
K_{\sF} \ = \ K_{X/Y}-R(\pi).
\end{equation}
\end{say}

\begin{rem}\label{rem_pullback}
Let $\varphi\colon X\map Y$ be a dominant rational map with connected fibers between normal varieties, and $\sF$  a foliation on $X$.
Suppose that the general fiber of $\varphi$ is tangent to $\sF$. This means that, for a general point $x$ on a general fiber $F$ of $\varphi$,
the linear subspace $\sF_x\subset T_xX$ determined by the inclusion $\sF\subset T_X$ 
contains $T_xF$. Then, by \cite[Lemma 6.7]{fano_fols}, there is a  foliation $\sG$ on $Y$
such that $\sF=\varphi^{-1}\sG$. 
We remark that this is not true in general if $\sF$ is just a distribution.
\end{rem}

\begin{say}[Restricting foliations to subvarieties] \label{restricting_fols}
Let $X$ be a smooth variety, and $\sF$  a codimension $q$ foliation on $X$.
Let $Z$ be a smooth subvariety with $\dim Z \ge q$.
Suppose that $Z$ is generically transverse to $\sF$. This means that the associated twisted $q$-form $\omega_{\sF}\in H^0\big(X,\Omega^{q}_X\otimes \det(\sN_\sF)\big)$ restricts to a nonzero twisted $q$-form on $Z$, and so $\sF$ induces a foliation $\sF_Z$ of codimension $q$ on $Z$.
Then there is an effective divisor $B$ on $Z$ such that $\sF_Z$ corresponds to a  twisted $q$-form in 
$H^0\big(Z,\Omega^{q}_Z\otimes \det(\sN_\sF)_{|Z}(-B)\big)$ non vanishing in codimension $1$.
A straightforward computation shows that 
$$
K_{\sF_Z}+{K_X}_{|Z} \ = \ {K_{\sF}}_{|Z}+K_Z-B.
$$
\end{say}

When $Z$ is a general hyperplane section, we show that $B=0$.

\begin{lemma}\label{bertini}
Let $X \subset \mathbb{P}^N$ be a smooth projective  variety, and $\sF$ a foliation  of codimension $q \le \dim X-2$ on $X$. 
Let $H \subset \mathbb{P}^N$ be a general hyperplane.
Then $\sF$ induces a  foliation $\sF_{X \cap H}$ of codimension $q$
on $X \cap H$ with 
$$
K_{\sF_{X \cap H}}=\big({K_\sF}+H\big)_{|X \cap H}.
$$
\end{lemma}

\begin{proof}
Let $H \subset \mathbb{P}^N$ be a general hyperplane, and set $Z:= X \cap H$. Notice that $Z$ is smooth and that $Z$ is generically transverse to $\sF$. 
Let $B$ be the effective divisor on $Z$ introduced in \ref{restricting_fols}. If $B_1$ is a prime divisor on $Z$, then $B_1 \subset \textup{Supp}(B)$ if and only if $\sF$ is tangent to $H$ at a general point of $B_1$.

Let $X^\circ \subset X$ be the open subset where $\sF$ is a subbundle of $T_X$, and consider the incidence subset
$$
I^\circ=\{(x,H)\in X^\circ \times (\mathbb{P}^N)^*\,|\, x\in H\text{ and }\sF_x \subset T_xH\}.
$$
An easy dimension count gives that $\dim I^\circ = \dim X + N - (\dim X - q) - 1 \le \dim X + N - 3$.
It follows that if $H$ is general, then $B=0$.
\end{proof}


\subsection{Algebraically integrable foliations}

Throughout this subsection, we
let $X$ be a normal projective variety, and $\sF$ a $\mathbb{Q}$-Gorenstein algebraically integrable foliation on $X$.

\begin{defn}[{\cite[Definition 3.11]{codim_1_del_pezzo_fols}}]\label{log_leaf} 
Let $i\colon F\to  X$ be the normalization of the closure of a general leaf of $\sF$. 
There is a  canonically defined effective $\mathbb{Q}$-divisor $\Delta_F$ on $ F$ such that
$K_{ F}  +  \Delta_F \sim_\mathbb{Q} i^*K_{\sF} $. 
The pair $(F,\Delta_F)$ is called a \emph{general log leaf} of $\sF$.  
\end{defn}

In the setup of Definition~\ref{log_leaf}, we often write $L_{|F}$ for the pull-back $i^*L$ of a Cartier divisor $L$ on $X$.

\begin{say}[{The family of log leaves of $\sF$  \cite[Lemma 3.9 and Remark 3.12]{codim_1_del_pezzo_fols}}] \label{family_leaves} 
There is a unique proper subvariety $Y'$  of the Chow variety of $X$ 
whose general point parametrizes the closure of a general leaf of $\sF$
(viewed as a reduced and irreducible cycle in $X$). 
Let $Y$ be the normalization of $Y'$, and  $Z \to Y'\times X$  the normalization of the universal cycle, with induced morphisms:
\begin{equation} \label{eq:family_of_leaves}
\xymatrix{
Z \ar[r]^{\nu}\ar[d]_{\pi} & X \\
 Y. &
}
\end{equation}
Then $\nu\colon Z\to X$ is birational and, for a general point $y\in Y$, 
$\nu\big(\pi^{-1}(y)\big) \subset X$ is the closure of a leaf of $\sF$.
We refer to the diagram \eqref{eq:family_of_leaves} as the \emph{family of leaves} of $\sF$.

Let $\sF_Z$ be the foliation on $Z$ induced by $\sF$ (or $\pi$).
By \eqref{morphism_fol}, $K_{\sF_Z}=K_{Z/Y}-R(\pi)$.
Moreover, there is a canonically defined effective Weil $\bQ$-divisor $\Delta$ on $Z$ such that 
\begin{equation}\label{eq:universal_canonical_bundle_formula}
K_{\sF_Z}+\Delta=K_{Z/Y}-R(\pi)+\Delta \sim_\mathbb{Q} \nu^* K_\sF.
\end{equation}
Note that $\Delta$ is $\nu$-exceptional since $\nu_*K_{\sF_Z}=K_\sF$. 

Let $y\in Y$ be a general point, set $Z_y := \pi^{-1}(y)$ and $\Delta_y:=\Delta_{|Z_y}$.
Then $(Z_y, \Delta_y)$ coincides with the general log leaf $(F,\Delta_F)$ from Definition~\ref{log_leaf}.
\end{say}

We will need the following observation.

\begin{lemma}\label{lemma:dicritical}
Let $X$ be a smooth projective variety, and $\sF$ an algebraically integrable foliation on $X$.
In the setup of Paragraph \ref{family_leaves} above, we have 
$$
\textup{Supp}(\Delta_y) = \textup{Exc}(\nu) \cap Z_y.
$$
In particular, the singular locus of $Z_y$ is contained in $\textup{Supp}(\Delta_y)$.
\end{lemma}

\begin{proof}
Since $\Delta$ is $\nu$-exceptional and $\textup{Supp}(\Delta_y)=\textup{Supp}(\Delta)\cap Z_y$, we must have
$\textup{Supp}(\Delta_y) \subset \textup{Exc}(\nu) \cap Z_y$.

To prove that $\textup{Exc}(\nu) \cap Z_y \subset \textup{Supp}(\Delta_y)$, we first reduce to the case when 
$\sF$ has rank 1.
Suppose that $\sF$ has rank $r \ge 2$, and consider an embedding $X \subset \mathbb{P}^N$.
Let $L \subset \mathbb{P}^N$ be a general linear subspace of codimension $r-1$.
By Lemma \ref{bertini}, $\sF$ induces a foliation by curves $\sF_{X \cap L}$ on $X \cap L$ with
	\begin{equation}\label{eq:adjunction_fol}
	K_{\sF_{X \cap L}}=\big({K_\sF}+(r-1)H\big)_{|X \cap L} \ ,
	\end{equation}
where $H$ is a hyperplane in $\mathbb{P}^N$. 
Notice that  the closure of a general leaf of 
$\sF_{X \cap L}$ is $\nu(Z_y)\cap L$ for a general point $y \in Y$.
Moreover, the general log leaf of $\sF_{X \cap L}$ is $\big(Z_y\cap \nu^{-1}(L),\Delta_y\cap \nu^{-1}(L)\big)$.
This follows from the definition of general log leaf, \eqref{eq:adjunction_fol}, and the usual adjunction formula. 
Let $E$ be an irreducible component of  $\textup{Exc}(\nu)$ meeting $Z_y$. 
Since the restriction of $\nu$ to $E\cap Z_y$ is finite, we have $\dim \nu(E) \ge r-1$.
In particular, $\nu(E) \cap L \neq \emptyset$. Moreover,
$\textup{Exc}(\nu) \cap Z_y \subset \textup{Supp}(\Delta_y)$
if and only if 
$\textup{Exc}(\nu) \cap Z_y \cap \nu^{-1}(L) \subset \textup{Supp}(\Delta_y) \cap \nu^{-1}(L)$.
Since $\textup{Exc}(\nu)\cap \nu^{-1}(L) = \textup{Exc}\big(\nu_{| \nu^{-1}(L)}\big)$, we may assume that $\sF$ has rank $1$.

When $\sF$ has rank $1$, $\sF$ is a line bundle on $X$ by \cite[Remark 2.3]{fano_fols} and \cite[Proposition 1.9]{hartshorne80}. 
If $E$ is an irreducible component of 
$\textup{Exc}(\nu)$ that dominates $Y$, then $\nu(E)$ is contained in the singular locus of $\sF$.
Apply \cite[Lemme 1.2]{druel04} (see also \cite[Lemma 5.6]{fano_fols}) to conclude 
that $E \cap Z_y \subset \textup{Supp}(\Delta_y)$, completing the proof of the lemma.
\end{proof}

\begin{cor}\label{delta_neq_0}
Let $X$ be a smooth projective variety, and $\sF\subsetneq T_X$ an algebraically integrable foliation on $X$, 
with general log leaf  $(F,\Delta_F)$. 
Suppose that either $\rho(X)=1$, or $\sF$ is a Fano foliation. 
Then $\Delta_F \neq 0$. 
\end{cor}

\begin{proof}
Let the notation be as in  Paragraph \ref{family_leaves}, and suppose that $\Delta_F = 0$. 
It follows from Lemma~\ref{restricting_fols} that no irreducible component of $\textup{Exc}(\nu)$ dominates $Y$. 
Hence, $\sF$ is induced by a rational map $X\map Y$  that restricts to a smooth proper morphism on a dense open subset of $X$.
This is impossible if $\rho(X)=1$.
Moreover, $F$ is smooth, and in particular log canonical. 
On the other hand, an algebraically integrable Fano foliation whose general log leaf is 
log canonical has the special property that there is a common point in the closure of a general leaf
(see \cite[Proposition 5.3]{fano_fols} and \cite[Proposition 3.13]{fano_fols_2}). 
We conclude that $\Delta_F \neq 0$. 
\end{proof}

We end this subsection with a consequence of \cite[Lemma 2.14]{hoering} 
(see also \cite{campana_paun}, \cite{campana_paun15} and \cite{druel15} for related results). 
Recall that a Weil $\mathbb{Q}$-divisor $D$ on a normal projective variety $X$ is said to be \textit{pseudo-effective} if, for any ample divisor $L$ on $X$ and any rational number 
$\varepsilon > 0$, there exists an effective Weil $\mathbb{Q}$-divisor $E$ such that
$D + \varepsilon L \sim_\mathbb{Q} E$.

\begin{lemma}\label{lemma:pseudo_effective}
Let $X$ be a normal projective variety, $L$ an ample divisor on $X$,
and  $\sF$ a $\mathbb{Q}$-Gorenstein algebraically integrable 
foliation on $X$. Let $F$ be the normalization of the closure of a general leaf of $\sF$, and let $\nu_F\colon \widehat{F}\to F$ be a resolution of singularities. If $K_{\widehat{F}}+\nu_F^*L_{|F}$ is pseudo-effective, then 
so is $K_\sF+L$.
\end{lemma}

\begin{proof}
Consider the family of leaves of $\sF$ as in Paragraph~\ref{family_leaves}:
$$
\xymatrix{
Z \ar[r]^{\nu}\ar[d]_{\psi} & X \\
 Y. &
}
$$
By \cite[Lemma 4.2]{druel15},
there exists a finite surjective morphism
$\mu_1\colon Y_1 \to Y$ with $Y_1$ normal and connected satisfying the following property. If $Z_1$ denotes the normalization of the product $Y_1 \times_Y Z$, then the induced morphism $\psi_1\colon Z_1 \to Y_1$ has reduced fibers over codimension one points in $Y_1$. Let $\nu_1\colon Z_1 \to Z$ be the natural morphism.
Let $\mu_2\colon Y_2\to Y_1$ be a resolution of singularities, and 
$Z_2$ a resolution of singularities of the product $Y_2 \times_{Y_1} Z_1$, with natural morphism $\nu_2\colon Z_2 \to Z_1$.
We have a commutative diagram:

\centerline{
\xymatrix{
Z_2 \ar[rr]^{\nu_2,\text{ birational}}\ar[d]^{\psi_2} && 
Z_1 \ar[rr]^{\nu_1,\text{ finite}}\ar[d]^{\psi_1} && Z \ar[d]^{\psi}\ar[rr]^{\nu,\text{ birational}} && X \\
Y_2 \ar[rr]_{\mu_2,\text{ birational}} && Y_1 \ar[rr]_{\mu_1,\text{ finite}} && Y. && \\
}
}
\noindent Let $F_1\cong F$ be a general fiber of $\psi_1$, and set $L_1:=(\nu\circ\nu_1)^*L$. 
Let $F_2$ be the fiber of $\psi_2\colon Z_2 \to Y_2$ mapping to $F_1$, and set $L_2:=\nu_2^*L_1$.

By Paragraph~\ref{family_leaves},
there is a canonically defined $\nu$-exceptional effective $\mathbb{Q}$-Weil divisor $\Delta$ on $Z$ such that 
\begin{equation}\label{eq:eq1}
K_{Z/Y}-R(\psi)+\Delta\sim_\bQ \nu^* K_\sF,
\end{equation}
where $R(\psi)$ denotes the ramification divisor of $\psi$. Moreover,
a straightforward computation shows that 
\begin{equation}\label{eq:eq2}
\nu_1^*\big(K_{Z/Y}-R(\psi)\big)=K_{Z_1/Y_1}.
\end{equation}

Suppose that $K_{\widehat{F}}+\nu_F^*L_{|F}$ is pseudo-effective.
We may assume without loss of generality that the restriction of
$\nu_2\colon Z_2\to Z_1$ to $F_2$
factors through $\widehat{F}\to F\cong F_1$. This implies that
$K_{F_2}+{L_2}_{|F_2}$ is pseudo-effective as well. 
By \cite[Lemma 2.14]{hoering},  
$K_{Z_2/Y_2}+L_2$ is also pseudo-effective.
This implies that so is $(\nu\circ\nu_1\circ\nu_2)_*(K_{Z_2/Y_2}+L_2)$. 
Set $Z_1^\circ:=Z_1\setminus \nu_2\big(\textup{Exc}(\nu_2)\big)$
and $Y_1^\circ:=Y_1\setminus \mu_2\big(\textup{Exc}(\mu_2)\big)$.
Since $\psi_1$ is equidimensional, we have $\psi_1(Z_1^\circ)\subset Y_1^\circ$.

Notice that $\codim \, Y_1 \setminus Y_1^\circ \ge 2$,  $\codim \, Z_1 \setminus Z_1^\circ \ge 2$, and 
that $\nu_2$ (respectively, $\mu_2$) induces an isomorphism
$\nu_2^{-1}(Z_1 \setminus Z_1^\circ) \cong Z_1 \setminus Z_1^\circ$
(respectively, $\mu_2^{-1}(Y_1 \setminus Y_1^\circ) \cong Y_1 \setminus Y_1^\circ$).
It follows that $(\nu_2)_*(K_{Z_2/Y_2}+L_2)=K_{Z_1/Y_1}+L_1$.
From \eqref{eq:eq1} and \eqref{eq:eq2}, we conclude that 
$$
(\nu\circ\nu_1\circ\nu_2)_*(K_{Z_2/Y_2}+L_2)=\deg(\nu_1)(K_\sF+L)
$$
is pseudo-effective, completing the proof of the lemma.
\end{proof}


\subsection{Foliations defined by stability conditions}\label{movable_classes}

Let $X$ be a normal projective variety, and $\sF\subset T_X$ a foliation on $X$.
Harder-Narasimhan filtrations of $\sF$ allow one to construct subfoliations of  $\sF$ that  inherit some of the positivity properties of 
$\sF$
(see for instance \cite[Section 7]{fano_fols}).
However, the  classical notion of slope-stability with respect to  an ample line bundle is not flexible enough to 
be applied in many situations in  birational geometry.
The papers \cite{campana_peternell11} and \cite{gkp_movable} extend a number of known results from the classical case
to the setting where stability conditions are given by movable curve classes.

\begin{defn}
A curve class $\alpha\in  \N_1(X)_\mathbb{R}$ is  \emph{movable} if 
$D\cdot \alpha \ge 0$ for all effective Cartier divisors $D$ on $X$.
The set of movable classes is a closed convex cone $\Mov(X)\subset \N_1(X)_\mathbb{R}$, called the
\emph{movable cone} of $X$.
If $X$ is smooth, then $\Mov(X)$ is the closure of the convex cone in $\N_1(X)_\mathbb{R}$
generated by classes of curves whose deformations cover a dense subset of $X$ by \cite{bdpp}.
\end{defn}

\begin{say}[The Harder-Narasimhan filtration with respect to a movable curve class]\label{HNF}
Let $X$ be a normal, $\mathbb{Q}$-factorial, projective variety, $\alpha\in  \N_1(X)_\mathbb{R}$ a movable curve class, and 
$\sF$ a torsion-free sheaf of positive rank on $X$.

The \emph{slope of $\sF$ with respect to $\alpha$} is the real number
	$$
	\mu_{\alpha}(\sF)=\frac{\det(\sF)\cdot \alpha}{\rank(\sF)}.
	$$  

The sheaf  $\sF$ is \emph{$\alpha$-semistable} if, for any  
subsheaf $\sE\neq 0$ of $\sF$, one has $\mu_{\alpha}(\sE)\le\mu_{\alpha}(\sF)$.

The maximal and minimal slopes of $\sF$ with respect to $\alpha$ are defined by 
	\begin{align*}
	&\mu_\alpha^{\max}(\sF):=\sup \big\{\mu_\alpha(\sE)\,|\,0\neq \sE \subseteq \sF\text{ is a coherent subsheaf}\big\},  \\
	&\mu_\alpha^{\min}(\sF):= \inf \big\{\mu_\alpha(\cQ)\,|\, \cQ\neq 0  \text{ is a  torsion-free quotient of }  \sF \big\}. 
	\end{align*}

By \cite[Corollary 2.26]{gkp_movable}, there exists a unique filtration
of $\sF$ by saturated subsheaves
	$$
	0=\sF_0\subsetneq \sF_1\subsetneq \cdots\subsetneq \sF_k=\sF,
	$$
with $\alpha$-semistable quotients $\cQ_i=\sF_i/\sF_{i-1}$ 
such that $\mu_{\alpha}(\cQ_1) > \mu_{\alpha}(\cQ_2) > \cdots > \mu_{\alpha}(\cQ_k)$.
This filtration is called the \emph{Harder-Narasimhan filtration} of $\sF$. 

Using the basic properties of slopes and the Harder-Narasimhan filtration, one can check  that, for $1\le i\le k$,
\begin{equation}\label{mu_min_Fi}
\mu_\alpha^{\min}(\sF_i) \ = \ \mu_{\alpha}(\cQ_i) \ = \ \mu_\alpha^{\max}(\sF/\sF_{i-1}).
\end{equation}

The sheaf $\sF_1$ is called the \emph{maximal destabilizing subsheaf} of $\sF$. 

Suppose that $\mu_\alpha^{\max}(\sF)>0$, and set $s:=\max\{1\le i\le k\,|\,\mu_{\alpha}(\cQ_i)>0\}\ge 1$.
The \emph{positive part of $\sF$ with respect to $\alpha$} is the sheaf $\sF_\alpha^+:=\sF_{s}$.
\end{say}

The following is a useful criterion for a subsheaf of a foliation to be a  foliation. We include a proof for the reader's convenience. 

\begin{lemma}[{\cite[Lemma 4.12]{campana_paun15}}]\label{CP15_lemma_4.9}
Let $X$ be a normal, $\mathbb{Q}$-factorial, projective variety and  $\sF\subset T_X$ a foliation on $X$.
Let $\sG\subset \sF$ be a saturated subsheaf and suppose that, for some movable curve class $\alpha\in  \N_1(X)_\mathbb{R}$,
$$
2 \mu_\alpha^{\min}(\sG) > \mu_\alpha^{\max}(\sF/\sG).
$$
Then $\sG$ is also a foliation on $X$.
\end{lemma}

\begin{proof}Note that $\sG$ is saturated in $T_X$.
Integrability of $\sG$ is equivalent to the vanishing of the map $\wedge^2\sG \ \to \ \sF/\sG$ induced by the Lie bracket.
This vanishing follows from the inequality
$$
\mu_\alpha^{\min}\big(\wedge^2\sG/Tors\big) > \mu_\alpha^{\max}(\sF/\sG).
$$
So it is enough to prove that
\begin{equation}\label{eq:mu_min_wedge}
\mu_\alpha^{\min}\big(\wedge^2\sG/Tors\big)\ge 2 \mu_\alpha^{\min}(\sG).
\end{equation}

Observe $\big(\wedge^2\sG/Tors\big)^{**}\cong (\wedge^2\sG)^{**}$ is a direct summand of 
$(\sG\otimes\sG)^{**}$. Therefore, we have 
\begin{equation*}
\begin{aligned}
\mu_\alpha^{\min}\big(\wedge^2\sG/Tors\big) \ \ &  = \ \ \mu_\alpha^{\min}\big((\wedge^2\sG)^{**}\big) & \\ 
& \ge \ \  \mu_\alpha^{\min}\big((\sG\otimes\sG)^{**}\big) &  \\ 
& = \ \  2\mu_\alpha^{\min}(\sG)  & (\text{by  \cite[Theorem 4.2]{gkp_movable}),} \\ 
\end{aligned} 
\end{equation*}
proving \eqref{eq:mu_min_wedge}.
\end{proof}

\begin{cor}\label{integrability}
In the setup of Paragraph~\ref{HNF}, suppose that $\sF\subset T_X$ is a foliation on $X$ with $\mu_\alpha^{\max}(\sF)\ge 0$.
Then $\sF_i\subset T_X$ is also a foliation on $X$ whenever $\mu_{\alpha}(\cQ_i)\ge 0$. 
\end{cor}

We end this subsection with a remarkable result of Campana and P\u{a}un
concerning algebraic integrability of foliations.

\begin{thm}\label{algebraic_integrability_criterion}
Let $X$ be a normal $\mathbb{Q}$-factorial projective variety, $\alpha\in  \N_1(X)_\mathbb{R}$ a movable curve class, and $\sF\subset T_X$ a foliation on $X$. Suppose that $\mu_\alpha^{\min}(\sF) > 0$. 
Then $\sF$ is algebraically integrable, and  the closure of a general leaf is rationally connected.

In particular, if  $\sF$ is purely transcendental, then $K_{\sF}$ is pseudo-effective. 
\end{thm}

\begin{proof}
Let $\nu\colon \widehat{X} \to X$ be a resolution of singularities. By \cite[Proposition 2.7 and Remark 2.8]{gkp_movable}, 
we have $\mu_\alpha^{\min}(\sF) = \mu_{\nu^*\alpha}^{\min}\big(\nu^{-1}\sF\big)$, where 
$\nu^*\alpha\in \Mov(\widehat{X})$ is the numerical pull-back of $\alpha$. 
The claim now follows from \cite[Theorem 4.2]{campana_paun15} applied to $\nu^{-1}\sF$.
\end{proof}

The following is an immediate consequence of Corollary~\ref{integrability} and Theorem~\ref{algebraic_integrability_criterion}. 

\begin{cor}\label{alg_integrability}
In the setup of Paragraph~\ref{HNF}, suppose that $\sF\subset T_X$ is a foliation on $X$ with $\mu_\alpha^{\max}(\sF)>0$.
Then, for $1\le i\le s$, $\sF_i\subset T_X$ is an  algebraically integrable foliation, and  the closure of a general leaf is rationally connected.
\end{cor}

%
%
%
%

\section{Characterization of generic projective space bundles}

\begin{thm}\label{thm:positive_twisted_slope}
Let $X$ be a normal $\mathbb{Q}$-factorial projective variety, and $\sL$ an ample line bundle on $X$.
Suppose that $\mu^{\max}_\alpha(T_X\otimes\sL^*)>0$ for some movable curve class 
$\alpha\in\Mov(X)$, and let $\sT_1$ be the maximal destabilizing subsheaf of $T_X$ with respect to $\alpha$.
\begin{enumerate}
\item Then $\sT_1$ is induced by a generic $\mathbb{P}^{r_1}$-bundle structure 
$\pi_1\colon X \map Y_1$ on $X$, and $\sL$ restricts to $\sO_{\mathbb{P}^{r_1}}(1)$ on a general fiber of $\pi_1$. Moreover, $\mu_{\alpha}^{\max}\big((T_X/\sT_1)\otimes\sL^*\big) \le 0$.
\item If $\sF$ is a foliation on $X$ and $\mu_\alpha(\sF\otimes\sL^*)>0$, then  $\sT_1 \subseteq \sF$, and equality holds if and only if $\sF$ is algebraically integrable.
\end{enumerate}
\end{thm}

\begin{proof}

\

\medskip

\noindent\textbf{Step 1}.
Suppose first that $\sF$ is an algebraically integrable foliation of rank $r$, 
and $\mu_\alpha(\sF\otimes\sL^*)>0$ for some movable curve class  $\alpha\in\Mov(X)$. 
We show that
$\sF$ is induced by a generic $\mathbb{P}^r$-bundle structure 
$\pi\colon X \map Y$ on $X$ and that $\sL$ restricts to $\sO_{\mathbb{P}^r}(1)$ on a general fiber of $\pi$.

Let $F$ be the normalization of the closure of a general leaf of $\sF$, and let 
$\nu_F\colon \widehat{F}\to F$ be a resolution of singularities. Let $L$ be a divisor on $X$ such that 
$\sO_X(L)\cong \sL$. 
If $K_{\widehat{F}}+\dim F\cdot\nu_F^*L_{|F}$ is pseudo-effective, then so is 
$K_\sF+\dim F\cdot L$ by Lemma \ref{lemma:pseudo_effective}.
This is absurd since $-r\mu_\alpha(\sF\otimes\sL^*)=(K_\sF+\dim F\cdot L)\cdot\alpha<0$ by assumption, proving 
that $K_{\widehat{F}}+\dim F\cdot\nu_F^*L_{|F}$ is not pseudo-effective.
Apply \cite[Lemma 2.5]{hoering_fol} to conclude that 
$(F,\sL_{|F})\cong(\mathbb{P}^r,\sO_{\mathbb{P}^r}(1))$.
Consider the family of leaves of $\sF$ as in Paragraph~\ref{family_leaves}:
$$
\xymatrix{
Z \ar[r]^{\nu}\ar[d]_{\psi} & X \\
Y. &
}
$$ 
By \cite[Proposition 4.10]{codim_1_del_pezzo_fols}, we have
$(Z,\nu^*\sL)\cong(\mathbb{P}_Y\big(\sE),\sO_{\mathbb{P}_Y(\sE)}(1)\big)$ where
$\sE:=\psi_*\nu^*\sL$. 

Suppose that the exceptional set
$\textup{Exc}(\nu)$ dominates $Y$ under $\psi$. 
Consider a resolution of singularities
$\mu\colon \widehat{Y} \to Y$, and set 
$\widehat{Z}=\mathbb{P}_{\widehat{Y}}(\mu^*\sE)\cong \widehat{Y}\times_YZ$ with natural morphisms
$\widehat{\psi}\colon\widehat{Z}\to \widehat{Y}$ and 
$\widehat{\nu}\colon\widehat{Z}\to X$.
The exceptional set
$\textup{Exc}(\widehat{\nu})$ also dominates $\widehat{Y}$ under $\widehat{\psi}$.
Note that $\textup{Exc}(\widehat{\nu})$ has pure codimension one since $X$ is $\mathbb{Q}$-factorial.
Let $\widehat{E}$ be an irreducible component of 
$\textup{Exc}(\widehat{\nu})$ dominating $\widehat{Y}$.
Then $\sO_{\widehat{Z}}(-\widehat{E})\cong \sO_{\mathbb{P}_{\widehat{Y}}(\mu^*\sE)}(-k)\otimes\widehat{\psi}^*\sI$
for some positive integer $k$, and some line bundle $\sI$ on $\widehat{Y}$. In particular, 
$h^0(\widehat{Y},\textup{S}^k\mu^*\sE\otimes \sI^*)\ge 1$. 
Let $\widehat{\xi}$ be a divisor on $\widehat{Z}$ such that
$\sO_{\widehat{Z}}(\widehat{\xi})\cong \sO_{\mathbb{P}_{\widehat{Y}}(\mu^*\sE)}(1)$. Notice that 
$\sO_{\mathbb{P}_{\widehat{Y}}(\mu^*\sE)}(1)$ is the pull-back of $\sO_{\mathbb{P}_Y(\sE)}(1)\cong\nu^*\sL$ under the natural morphism 
$\mathbb{P}_{\widehat{Y}}(\mu^*\sE) \to \mathbb{P}_Y(\sE)\cong Z$, and hence it is semi-ample. 
From Lemma \ref{lemma:pseudo-effective_cone} below, we conclude that 
$m_0(K_{\widehat{Z}/\widehat{Y}}+r\widehat{\xi}+s\widehat{E})$
is effective for some positive integer $m_0$ and every sufficiently large integer $s$.
This in turn implies that 
$$
\widehat{\nu}_*(K_{\widehat{Z}/\widehat{Y}}+r\widehat{\xi}+s\widehat{E})=K_{\sF}+rL
$$
is pseudo-effective. Thus $\mu_\alpha(\sF\otimes\sL^*) \le 0$, yielding a contradiction.
This proves that the map $\pi\colon X\map Y$ induced by $\psi$ is almost proper, and so 
$\sF$ is induced by a generic $\mathbb{P}^r$-bundle structure  on $X$.

\medskip

\noindent\textbf{Step 2}. We prove statement (1).

Let 
$
0=\sT_0\subsetneq \sT_1\subsetneq \cdots\subsetneq \sT_k=T_X
$
be the Harder-Narasimhan filtration of $T_X$ with respect to $\alpha$, and set $\cQ_i=\sT_i/\sT_{i-1}$. 
By Corollary~\ref{alg_integrability}, $\sT_1$ is an algebraically integrable foliation on $X$. 
Moreover,  $\mu_\alpha(\sT_1\otimes\sL^*)>0$. From Step 1 applied to $\sT_1$, we conclude that  
$\sT_1$ is induced by a generic $\mathbb{P}^{r_1}$-bundle structure 
$\pi_1\colon X \map Y_1$ on $X$ and $\sL$ restricts to $\sO_{\mathbb{P}^{r_1}}(1)$ on a general fiber of $\pi_1$.

Suppose that $\mu^{\max}_\alpha\big((T_X/\sT_1)\otimes\sL^*\big)=\mu_\alpha(\cQ_2\otimes\sL^*)>0$. 
By Corollary~\ref{alg_integrability}, $\sT_2$ is an algebraically integrable foliation on $X$. 
By Step 1, $\sT_2$ is induced by a generic $\mathbb{P}^{r_2}$-bundle structure on $X$. This is impossible since $\mathbb{P}^{r_2}$ cannot admit a 
generic $\mathbb{P}^{r_1}$-bundle structure with $r_1<r_2$. Thus $\mu_\alpha(\sQ_2\otimes\sL^*) \le 0$,
proving (1).

\medskip

\noindent\textbf{Step 3}. We prove statement (2).

Suppose that $\sF$ is a foliation on $X$, and
$\mu_\alpha(\sF\otimes\sL^*)>0$ for some movable curve class $\alpha\in\Mov(X)$.
Let $\sF_1 \subset \sF$ be the maximal destabilizing subsheaf of $\sF$ with respect to $\alpha$.
By Corollary~\ref{alg_integrability}, $\sF_1$ is an algebraically integrable foliation on $X$. 
Moreover, we have $\mu_\alpha(\sF_1\otimes\sL^*)>0$.
By Step 1, $\sF_1$ is induced by generic $\mathbb{P}^{s_1}$-bundle structure on $X$.
Let $\sT_1\subset T_X$ be as in Step 2.
By \cite[Corollary 2.17]{gkp_movable}, $\sF_1 \subset \sT_1$ since 
$\mu_\alpha^{\textup{min}}(\sF_1\otimes\sL^*)=\mu_\alpha(\sF_1\otimes\sL^*)>0$
and $\mu_{\alpha}^{\max}\big((T_X/\sT_1)\otimes\sL^*\big) \le 0$ by Step 2. Hence, we must have $\sF_1 = \sT_1$.

If $\sF$ is algebraically integrable, then from Step 1, we conclude that 
$\sF_1 = \sT_1=\sF$. 
\end{proof}

Recall that a vector bundle $\sE$ on a normal projective variety $Y$ is said to be \emph{semi-ample} if the line bundle $\sO_{\mathbb{P}_Y(\sE)}(m)$ on $\mathbb{P}_Y(\sE)$ is generated by its global sections for some positive integer $m$.

\begin{lemma}\label{lemma:pseudo-effective_cone}
Let $Y$ be a normal projective variety,  $\sM$  a line bundle on $Y$, 
and  $\sE$  a semi-ample vector bundle on $Y$. Suppose that $h^0(Y,\textup{S}^k\sE\otimes \sM)\ge 1$ for some
integer $k \ge 1$. Then there exists an integer $m_0 \ge 1$ such that
$h^0(Y,\textup{S}^{m_0(sk-1)}\sE\otimes\det(\sE)^{\otimes m_0}\otimes \sM^{\otimes m_0s})\ge 1$ for every sufficiently large integer $s$.
\end{lemma}

\begin{proof}
Let $\sE$  be a semi-ample vector bundle on $Y$.
By \cite[Corollary 1]{fujiwara}, the vector bundle $\sE^*\otimes\det(\sE)$ is semi-ample. Let $m_0$ be a positive integer such that the line bundle 
$\sO_{Z}(m_0)$ on $Z:=\mathbb{P}_Y(\sE^*\otimes\det(\sE))$ is generated by its global sections. Set $V:=H^0\big(Y,\textup{S}^{m_0}(\sE^*\otimes\det(\sE))\big)\cong H^0\big(Z,\sO_Z(m_0)\big)$, and consider the exact sequence
$$
0 \to \sK\otimes \sO_Z(m) \to V \otimes \sO_Z(m) \to \sO_Z(m_0+m)\to 0.
$$
Pick a point $y$ on $Y$, and  denote by $Z_y$ the fiber over $y$ of the natural morphism $Z \to Y$. 
Let $m_1$ be a positive integer such that 
$h^1\big(Z_y,\sK_{|Z_y}\otimes\sO_{Z_y}(m)\big)=0$ for $m\ge m_1$.
Then the map
$$
V\otimes H^0\big(Z_y,\sO_{Z_y}(m)\big) \to H^0\big(Z_y,\sO_{Z_y}(m_0+m)\big)
$$
is surjective for every $m\ge m_1$, and thus
the morphism 
$$
V\otimes\textup{S}^m\big(\sE^*\otimes\det(\sE)\big) \to \textup{S}^{m_0+m}\big(\sE^*\otimes\det(\sE)\big)
$$
is generically surjective. This yields,  for every $m\ge m_1$, an injective map of sheaves
\begin{equation}\label{into}
\textup{S}^{m_0+m}(\sE) \into V^*\otimes\textup{S}^m(\sE)\otimes\det(\sE)^{\otimes m_0}.
\end{equation}
Let $\sM$ be a line bundle on $Y$, and $k$ a positive integer such that 
$h^0(Y,\textup{S}^k\sE\otimes \sM)\ge 1$. 
Let $s$ be an integer such that $m:=m_0(sk-1) \ge m_1$. 
The $m_0s$-th power of a nonzero global section of $\textup{S}^k\sE\otimes \sM$ is a 
nonzero global section of $\textup{S}^{m_0sk}\sE\otimes\sM^{\otimes m_0s}$.
Then \ref{into} yields
$$
1\le h^0(Y,\textup{S}^{m_0sk}\sE\otimes\sM^{\otimes m_0s}) \le 
\dim V \cdot h^0(Y,\textup{S}^{m_0(sk-1)}\sE\otimes\det(\sE)^{\otimes m_0}\otimes\sM^{\otimes m_0s}),
$$
completing the proof of the lemma.
\end{proof}

\begin{proof}[{Proof of Corollary \ref{cor:p^-bdles_singular_case}}]
Suppose that there is a positive integer $m$, an ample line bundle $\sL$, 
a torsion-free  quotient  $(\Omega_X^1\otimes\sL)^{\otimes m}\twoheadrightarrow\sQ$ of positive rank,
and a movable class $\alpha\in\Mov(X)$ such that $\mu_\alpha(\sQ)<0$.
Then $\mu_\alpha^{min}\big((\Omega_X^1\otimes\sL)^{\otimes m}\big)<0$.
This implies that $\mu_\alpha^{max}\big((T_X\otimes\sL^*)^{\otimes m}\big)>0$.
Thus, by \cite[Theorem 4.2]{gkp_movable} (see also \cite[Proposition 6.1]{campana_peternell11}), $\mu_\alpha^{max}\big(T_X\otimes\sL^*\big)>0$.
Theorem \ref{thm:positive_twisted_slope} then implies that $X$ is a generic $\mathbb{P}^r$-bundle for some positive integer $r$.
\end{proof}

%
%
%
%

\section{Bounding the algebraic rank}

We start this section by justifying the definition of  generalised index in Definition~\ref{defn:generalized_index}.

\begin{lemma}\label{lemma:generalized_index}
Let $X$ be a complex projective manifold, and $D$ a big divisor on $X$.
Set 
$$\widehat{\iota}(D):=\sup\big\{t\in\mathbb{R} \ | \ D \equiv tA+E \text{ where } A \text{ is an ample divisor and }
E \text{ is a pseudo-effective }\mathbb{R}\text{-divisor}\big\}.
$$
Then there exists an ample divisor $A_0$ and a pseudo-effective $\mathbb{R}$-divisor $E_0$ on $X$ such that 
$D \equiv \widehat{\iota}(D)A_0+E_0$.
\end{lemma}

\begin{proof}
Set $t_0=\frac{\widehat{\iota}(D)}{2}<\infty$.
Pick a real number $t\ge t_0$, an ample divisor $A_t$ and a pseudo-effective $\mathbb{R}$-divisor $E_t$  on $X$ such that $D \equiv tA_t+E_t$. 
Then $\frac{1}{t_0}D-A_t=(\frac{t}{t_0}-1)A_t+\frac{1}{t_0}E_t$  is pseudo-effective.
In order to prove the lemma, it is enough to show that 
there are finitely many  classes of integral effective divisors $B$ on $X$ such that $\frac{1}{t_0}D-B$  is pseudo-effective.
Let $C_1,\ldots,C_m$ be movable curves on $X$ such that $[C_1],\ldots,[C_m]$ is a basis of $\N_1(X)_\mathbb{R}$. 
If  $B$ is an effective divisor on $X$ such that $\frac{1}{t_0}D-B$ is pseudo-effective, then we have
$0 \le B \cdot C_i \le \frac{1}{t_0}D\cdot C_i$. 
These inequalities define a compact set $\Delta \subset \N^1(X)_\mathbb{R}$. 
Since the set of classes of effective divisors is discrete in $\N^1(X)_\mathbb{R}$, 
the compact set $\Delta$ contains at most finitely many of these classes.
\end{proof}

Next we show that the algebraic rank of a foliation is bounded from below by its generalised index, and classify
foliations attaining this bound (Theorem~\ref{thm:ko}). 
We fix the notation to be used in what follows.

\begin{notation}\label{notation:sect4}
Let $\sF\subseteq T_X$ be a foliation  with big anti-canonical class and generalised index $\widehat{\iota}$.
Denote by $\sF^a$ the algebraic part of $\sF$, and by $r^a$ its algebraic rank.
By  Lemma~\ref{lemma:generalized_index}, there is an ample divisor $L$ and a pseudo-effective $\mathbb{R}$-divisor $E$ 
such that $-K_\sF\equiv \widehat{\iota} L + E$.
Set $\sL:=\sO_X(L)$, and $\alpha_0=[\sL^{n-1}] \in\Mov(X)$.
For any movable curve class $\alpha \in\Mov(X)$, we denote by $\sF_\alpha^+$
the positive part of $\sF$ with respect to $\alpha$, as defined in Paragraph~\ref{HNF}, 
and by $r_\alpha=\rank (\sF_{\alpha}^+)$ its rank. 
\end{notation}

\begin{proof}[Proof of Theorem~\ref{thm:ko}]
We follow Notation~\ref{notation:sect4}, and assume that $\sF\neq T_X$.

For any movable curve class $\alpha \in\Mov(X)$ such that $\mu_{\alpha}(\sF_{\alpha}^+\otimes\sL^*)\le 0$, we have 
\begin{equation}\label{rankF+}
\begin{aligned}
\widehat{\iota}\mu_{\alpha}(\sL) \ \ &  \le \ \ \det(\sF)\cdot\alpha & \\ 
& \le \ \  \det(\sF_{\alpha}^+)\cdot\alpha & (\text{ since } \mu_{\alpha}(\sF/\sF_{\alpha}^+) \le 0)\  \\ 
& \le \ \  r_\alpha  \mu_{\alpha}(\sL) & (\text{ since } \mu_{\alpha}(\sF_{\alpha}^+\otimes\sL^*)\le 0)\  \\ 
& \le \ \  r^a \mu_{\alpha}(\sL) & (\text{ since } \sF_{\alpha}^+\subset \sF^a \text{ by Theorem \ref{algebraic_integrability_criterion}}). 
\end{aligned} 
\end{equation}
By Theorem \ref{thm:hoering}, we have  $\mu_{\alpha_0}(\sF_{\alpha_0}^+\otimes\sL^*)\le 0$  for the ample class  $\alpha_0$, and 
\eqref{rankF+}  gives that  
$\widehat{\iota} \le r^a$.

\medskip

Suppose from now on that $\widehat{\iota} = r^a$. 
Then \eqref{rankF+} shows that, for any movable curve class $\alpha \in\Mov(X)$ such that $\mu_{\alpha}(\sF_{\alpha}^+\otimes\sL^*)\le 0$
(in particular for $\alpha=\alpha_0$), we have $\sF^a=\sF_{\alpha}^+$, and $\mu_{\alpha}(\sF^a\otimes\sL^*) =  0$. Moreover $E \cdot \alpha_0=0$,
and thus $E=0$ and $-K_\sF\equiv \widehat{\iota} L$. 

Next we show that $\mu_{\alpha}(\sF^a\otimes\sL^*) =  0$  for every $\alpha\in\Mov(X)$. 
Suppose that this is not the case. Since the ample class  $\alpha_0$ lies in the interior of $\Mov(X)$ and $\mu_{\alpha_0}(\sF^a\otimes\sL^*) =  0$,
there must exist a class $\alpha\in\Mov(X)$ such that  $\mu_{\alpha}(\sF^a\otimes\sL^*)> 0$.
By Theorem \ref{thm:positive_twisted_slope},
there is a dense open  subset $X^\circ \subset X$ and a 
$\mathbb{P}^{r^a}$-bundle structure $\pi^\circ\colon X^\circ \to Y^\circ$ such that 
$\sF^a_{|X^\circ}=T_{X^\circ/Y^\circ}$, and  $\sL$ restricts to $\sO_{\mathbb{P}^{r^a}}(1)$ on any fiber of $\pi^\circ$. 
By Remark~\ref{rem_pullback}, there is a foliation $\sG$ on $Y^\circ$ such that 
$\sF_{|X^\circ}=({\pi^\circ})^{-1}\sG$. 
By shrinking $Y^\circ$ if necessary, we may assume that $\sG$ is locally free,
and  $\sF/T_{X^\circ/Y^\circ} \cong {\phi^*\sG}$.
Therefore, for a general fiber $F$ of $\pi^\circ$, we have 
$$
\sO_{\mathbb{P}^{r_a}}(\widehat{\iota})\cong \det(\sF)_{|F} \cong \det(T_{X^\circ/Y^\circ})_{|F} \cong 
\sO_{\mathbb{P}^{r_a}}(r^a+1),
$$
contradicting our assumption that  $\widehat{\iota} = r^a$. 
We conclude that $\mu_{\alpha}(\sF^a\otimes\sL^*) =  0$  for every $\alpha\in\Mov(X)$.

This implies that $-K_{\sF^a}\equiv r^a L$, and hence $\sF^a$ is an algebraically integrable Fano foliation. 
Let $(F,\Delta_F)$ be a general log leaf of $\sF^a$, so that
$-(K_{ F}  +  \Delta_F )\equiv r^a L_{|F}$.
By Corollary~\ref{delta_neq_0}, $\Delta_F \neq 0$, and thus $K_{ F}+r^a L_{|F} \equiv -\Delta_F$ is not pseudo-effective. 
By \cite[Lemma 2.5]{hoering_fol},
$\big(F,\sL,\sO_F(\Delta_F)\big)\cong\big(\mathbb{P}^{r^a},\sO_{\mathbb{P}^{r^a}}(1),\sO_{\mathbb{P}^{r^a}}(1)\big)$.
In particular, $(F,\Delta_F)$ is log canonical. By \cite[Proposition 5.3]{fano_fols}, there is a common point $x$ in the closure of a general leaf of $\sF^a$. 
This implies that $X$ is covered by rational curves passing through $x$ having degree $1$ with respect to $\sL$. 
It follows that $X$ is a Fano manifold with Picard number $\rho(X)=1$, and hence $-K_{\sF^a}\equiv r^a L$ implies that 
$-K_{\sF^a}\sim r^a L$.
By \cite[Theorem 1.1]{adk08}, $X \cong \mathbb{P}^n$ and $\sF^a$ is induced by a linear
projection $\pi\colon \mathbb{P}^n \dashrightarrow \mathbb{P}^{n-r_a}$. 
By Remark~\ref{rem_pullback}, there is a foliation $\sG$ on $\mathbb{P}^{n-r_a}$ such that $\sF=\pi^{-1}\sG$. 
By \eqref{pullback_fol} and \eqref{morphism_fol}, outside the center of the linear projection $\pi\colon \mathbb{P}^n \dashrightarrow \mathbb{P}^{n-r_a}$, 
we have $\pi^*K_\sG\sim K_\sF-K_{\sF^a}\sim 0$, and thus $K_\sG\sim 0$.
\end{proof}

Finally, we describe foliations $\sF$  whose algebraic rank slightly exceeds the generalised index, namely
$ \widehat{\iota}(\sF) < r^a(\sF)\le\widehat{\iota}(\sF) +1$.
We start by addressing foliations $\sF$ with $K_{\sF}\equiv 0$ and  algebraic rank $r^a=1$. 

\begin{lemma}\label{lemma:alg_rank_one}
Let $X$ be a complex projective manifold with Picard number $\rho(X)=1$, and
let $\sF$ be a foliation on $X$. Suppose that $K_\sF\equiv 0$ and that $\sF$ has algebraic rank $1$.
Then the closure of a general leaf of the algebraic part of $\sF$ is a  rational curve.
\end{lemma}

\begin{proof}
There exist a normal projective variety $Y$, a dominant rational map $\pi\colon X \dashrightarrow Y$
of relative dimension $1$, and a purely transcendental
foliation $\sG$ on $Y$ such that $\sF=\pi^{-1}\sG$. 
Denote by $\sC$ the foliation induced by $\pi$, i.e., the algebraic part of $\sF$.

After replacing $Y$ with a birationally equivalent variety, we may assume that $Y$ is the family of leaves of $\sC$. 
Let $X^\circ \subset X$ be an open subset with complement of codimension at least $2$ such that $\pi$ restricts to an 
equidimensional morphism
$\pi_{|X^\circ}\colon X^\circ \to Y$.
By \eqref{pullback_fol} and \eqref{morphism_fol} applied to $\pi_{|X^\circ}$, there is an effective divisor $R$ on $X$ such that, on $X^\circ$, we have
$$
- K_\sC \equiv \pi^*K_\sG + R.
$$
By Theorem \ref{algebraic_integrability_criterion}, $K_\sG$ is pseudo-effective,  and hence
so is $-K_\sC$. 
If $K_\sC\not\equiv 0$, then the lemma follows from Theorem~\ref{algebraic_integrability_criterion}.
Suppose from now on that $K_\sC\equiv 0$, and let 
$(F,\Delta_F)$ be a general log leaf of $\sC$, so that $K_F+\Delta_F \equiv 0$. 
By Corollary~\ref{delta_neq_0}, $\Delta_F\neq 0$.
Thus $\deg(K_F)<0$, and hence $F\cong\mathbb{P}^1$.
\end{proof}

\begin{proof}[Proof of Theorem~\ref{thm:rat_connectedness_leaves}]
We follow Notation~\ref{notation:sect4}, and assume that $ \widehat{\iota} < r^a\le\widehat{\iota} +1$.
We will show that the closure of a general leaf of $\sF^a$ is rationally connected.

\medskip

\noindent\textbf{Step 1}. 
Let $\alpha \in\Mov(X)$ be a movable curve class. 
If $\mu_{\alpha}(\sF_{\alpha}^+\otimes\sL^*) < 0$, then \eqref{rankF+} gives that 
$r^a-1\le\widehat{\iota} < r_\alpha \le r^a$. Hence, $\sF_{\alpha}^+=\sF^a$. 
By Theorem \ref{algebraic_integrability_criterion}, the closure of a general leaf of $\sF^a$ is rationally connected. 
If $\mu_{\alpha}(\sF_{\alpha}^+\otimes\sL^*) = 0$, then \eqref{rankF+} gives that  
$r^a-1\le\widehat{\iota} \le r_\alpha \le r^a$, and 
hence $\sF_{\alpha}^+$ has codimension at most $1$ in $\sF^a$. 
If  $\sF_{\alpha}^+=\sF^a$, then, as before, the closure of a general leaf of $\sF^a$ is rationally connected. 
Suppose that $\sF_{\alpha}^+$ has codimension $1$ in $\sF^a$.
Then $\widehat{\iota} = r^a-1$.
 Let $\beta \in\Mov(X)$ be another movable curve class. 
If $\sF_\beta^+ \not\subset \sF_\alpha^+$, then $\sF_{\alpha}^+ +\sF_{\beta}^+ = \sF^a$ and 
Theorem \ref{algebraic_integrability_criterion} applied to both $\alpha$ and $\beta$ gives that 
the closure of a general leaf of $\sF^a$ is rationally connected.

Therefore, from now on we may assume the following.

\medskip

\noindent\textit{Additional assumption:} For any $\alpha\in\Mov(X)$,  $\mu_{\alpha}(\sF_{\alpha}^+\otimes\sL^*) \ge 0$.
If $\alpha\in\Mov(X)$ is such that $\mu_{\alpha}(\sF_{\alpha}^+\otimes\sL^*) = 0$, 
then $\sF_{\alpha}^+$ has codimension $1$ in $\sF^a$, and $\sF_\beta^+ \subset \sF_\alpha^+$
for any $\beta\in \Mov(X)$.

\medskip

By Theorem \ref{thm:hoering}, we have  $\mu_{\alpha_0}(\sF_{\alpha_0}^+\otimes\sL^*)\le 0$  for the ample class  $\alpha_0$.
The additional assumption and  \eqref{rankF+}  for $\alpha=\alpha_0$ give that $r_{\alpha_0}=r^a-1=\widehat{\iota}$, 
$E=0$  and $-K_\sF\equiv \widehat{\iota} L$.

\medskip

\noindent\textbf{Step 2}. 
Suppose that $\mu_{\alpha}(\sF_{\alpha}^+\otimes\sL^*) = 0$ for every
$\alpha$ in a nonempty open subset $U$ of $\Mov(X)$. 
From the additional assumption, it follows that $\sF_\alpha^+=\sF_{\alpha_0}^+$ and $\mu_{\alpha}(\sF_{\alpha_0}^+\otimes\sL^*) = 0$
for every $\alpha\in U$.
Thus $-K_{\sF_{\alpha_0}}\equiv r_{\alpha_0} L$.
By Theorem \ref{thm:ko}, $X \cong \mathbb{P}^n$, and $\sF_{\alpha_0}^+$ is induced by 
a linear projection $\pi\colon \mathbb{P}^n \dashrightarrow \mathbb{P}^{n-r_{\alpha_0}}$.
Since $X \cong \mathbb{P}^n$, we must have $-K_\sF\sim \widehat{\iota} L\sim -K_{\sF_{\alpha_0}}$. 
By Remark~\ref{rem_pullback}, there is a foliation $\sG$ on  $\mathbb{P}^{n-r_{\alpha_0}}$ such that 
$\sF = \pi^{-1}\sG$, and  $K_\sG \sim 0$. 
Moreover, since $\sF_{\alpha_0}^+$ has codimension $1$ in $\sF^a$, $\sG$ has algebraic rank $1$. 
By Lemma~\ref{lemma:alg_rank_one}, the closure of a general leaf of the algebraic part of $\sG$ is a rational curve,
and thus the closure of a general leaf of $\sF^a$ is rationally connected.

\

\noindent\textbf{Step 3}. 
By Step 2,
we may assume that $\mu_{\alpha}(\sF_{\alpha}^+\otimes\sL^*) > 0$ for some $\alpha\in \Mov(X)$.
By Theorem \ref{thm:positive_twisted_slope},
$\sF_\alpha^+$ is induced by a generic $\mathbb{P}^{r_\alpha}$-bundle structure 
$\pi_\alpha\colon X \dashrightarrow Y_\alpha$
on $X$,
and $\sL$ restricts to $\sO_{\mathbb{P}^{r_\alpha}}(1)$ on general fibers of $\pi_\alpha$.
If necessary, we can replace $\alpha$ by the class of a line on a fiber of $\pi_\alpha$.
As in \eqref{rankF+}, we have
$$
\begin{aligned}
(r^a-1)\  \ \ = \ \ \widehat{\iota}\  \ \ &  = \ \ \det(\sF)\cdot\alpha \\
& \le \ \  \det(\sF_{\alpha}^+)\cdot\alpha  \\ 
& = \ \ (r_{\alpha}+1)\  , 
\end{aligned} 
$$
and thus $r_{\alpha}\ge r^a- 2$. 
Moreover, by Corollary \ref{cor:degree 0}, 
$-K_{\sF_\alpha^+}-r_\alpha L$ is not pseudo-effective.
Thus, there exists a nonempty open subset $V$ of $\Mov(X)$ such that, 
for every $\beta\in V$, 
$\mu_{\beta}(\sF_{\alpha}^+\otimes\sL^*) < 0$. The additional assumption implies in particular that 
$\sF_\beta^+\neq \sF_\alpha^+$ for every $\beta\in V$.

By Step 2,
we may assume that $\mu_{\beta}(\sF_{\beta}^+\otimes\sL^*) > 0$ for some $\beta\in V$.
As before, 
$\sF_\beta^+$ is induced by a generic $\mathbb{P}^{r_\beta}$-bundle structure 
$\pi_\beta\colon X \dashrightarrow Y_\beta$ on $X$,
$\sL$ restricts to $\sO_{\mathbb{P}^{r_\beta}}(1)$ on general fibers of $\pi_\beta$, and 
$r_\beta\ge r^a-2$.
Notice that
$\sF_\alpha^+ \cap \sF_\beta^+ = 0$ and $\sF_{\alpha}^+ +\sF_{\beta}^+ \subset \sF^a$.
If $\sF_{\alpha}^+ +\sF_{\beta}^+ = \sF^a$, then we conclude as before that the closure of a general leaf of $\sF^a$ is rationally connected.
So we may assume that $\sF_{\alpha}^+ +\sF_{\beta}^+ \subsetneq \sF^a$.
This can only happen if $(r_\alpha,r_\beta,r^a)=(1,1,3)$.

Let $H_\alpha$ and $H_\beta$ be the dominating unsplit families of rational curves on $X$ whose general members correspond to
lines on fibers of $\pi_\alpha$ and $\pi_\beta$, respectively. 
Denote by $\pi^\circ\colon X^\circ \to T^\circ$ the $(H_\alpha,H_\beta)$-rationally connected quotient of $X$
(see for instance \cite[6.4]{fano_fols} for this notion).
By \cite[Lemma 2.2]{adk08}, we may assume that 
$\codim \, X \setminus X^\circ \ge 2$, $T^\circ$ is smooth, and $\pi^\circ$ has irreducible and reduced fibers. 
Applying \cite[Corollaires 14.4.4 et 15.2.3]{ega28}, we see that $\pi^\circ$ is flat. This in turn implies that its fibers are Cohen-Macaulay since both $X$ and $T^\circ$ are smooth.
By \cite[Lemma 6.9]{fano_fols}, there is an inclusion $T_{X^\circ/T^\circ}\subset \sF^a_{|X^\circ}$.
If $T_{X^\circ/T^\circ} = \sF^a_{|X^\circ}$, then the closure of a general leaf of $\sF^a$ is rationally connected. 
So we may assume that $\dim T^\circ = \dim X^\circ  -2$. 
By Remark~\ref{rem_pullback}, there is a foliation
$\sG^\circ$ on $T^\circ$ such that $\sF_{|X^\circ}={\pi^\circ} ^{-1}\sG^\circ$.
By \eqref{pullback_fol}, we have
$$
-K_{X^\circ/T^\circ}= -{K_\sF}_{|X^\circ} + {\pi^\circ}^*K_{\sG^\circ}.
$$
Since $\codim \, X \setminus X^\circ \ge 2$, a general complete intersection curve in $X$ is contained in $X \setminus X^\circ$.
Let $C \to T^\circ$ be the normalization of a complete curve passing through a general point, denote by $X_C$ the fiber product 
$C \times_{T^\circ} X^\circ$, and by $\pi_C\colon X_C \to C$ the natural morphism.
By \cite[Corollaire 5.12.4]{ega24}, $X_C$ satisfies Serre's condition $S_2$. On the other hand, $X_C$ is smooth in codimension $1$ since the fibers of $\pi_C$ are reduced.
It follows that $X_C$ is a normal variety by Serre's criterion for normality. We have
$$
-K_{X_C/C}= {-K_\sF}_{|X_C} + {\pi_C}^*{K_{\sG^\circ}}_{|C}.
$$
If  $K_{\sG^\circ}\cdot C \ge 0$, then $-K_{X_C/C}$ is ample, contradicting
\cite[Theorem 3.1]{adk08}. 
Hence, $K_{\sG^\circ}\cdot C < 0$, and  \cite[Proposition 7.5]{fano_fols}
implies that the closure of a general leaf of $\sF$ is rationally connected.
\end{proof}

\begin{ques}
Is there a foliation $\sF$ with $ \widehat{\iota}(\sF)\not\in \mathbb{N}$ and
$ \widehat{\iota}(\sF) < r^a(\sF) < \widehat{\iota}(\sF) +1$?
From the proof of  Theorem~\ref{thm:rat_connectedness_leaves}, we see that in this case we must have 
$\sF_{\alpha}^+=\sF^a$ for every $\alpha \in\Mov(X)$ such that $\mu_{\alpha}(\sF_{\alpha}^+\otimes\sL^*) \le 0$.
\end{ques}


\providecommand{\bysame}{\leavevmode\hbox to3em{\hrulefill}\thinspace}
\providecommand{\MR}{\relax\ifhmode\unskip\space\fi MR }
\providecommand{\MRhref}[2]{%
  \href{http://www.ams.org/mathscinet-getitem?mr=#1}{#2}
}
\providecommand{\href}[2]{#2}

\end{document}